\documentclass[12pt, reqno]{amsart}
\usepackage{amsmath, amsthm, amscd, amsfonts, amssymb, graphicx, color}
\usepackage[bookmarksnumbered, colorlinks, plainpages]{hyperref}

\input{mathrsfs.sty}

\newtheorem{theorem}{Theorem}[section]
\newtheorem{lemma}[theorem]{Lemma}
\newtheorem{proposition}[theorem]{Proposition}
\newtheorem{corollary}[theorem]{Corollary}
\theoremstyle{definition}

\theoremstyle{remark}
\newtheorem{remark}[theorem]{Remark}
\numberwithin{equation}{section}
\begin{document}

\title {Reverses and variations of Heinz inequality}

\author[M. Bakherad, M.S. Moslehian]{Mojtaba Bakherad$^1$ and Mohammad Sal Moslehian$^2$}

\address{$^{1}$ Department of Pure Mathematics, Ferdowsi University of Mashhad, P. O. Box 1159, Mashhad 91775, Iran}
\email{Mojtaba.Bakherad@yahoo.com; bakherad@member.ams.org}

\address{$^2$ Department of Pure Mathematics, Center of Excellence in
Analysis on Algebraic Structures (CEAAS), Ferdowsi University of
Mashhad, P. O. Box 1159, Mashhad 91775, Iran}
\email{moslehian@um.ac.ir, moslehian@member.ams.org}


\begin{abstract}
Let $A, B$ be positive definite $n\times n$ matrices. We present several reverse Heinz type inequalities, in particular
\begin{align*}
\|AX+XB\|_2^2+ 2(\nu-1) \|AX-XB\|_2^2\leq \|A^{\nu}XB^{1-\nu}+A^{1-\nu}XB^{\nu}\|_2^2,
 \end{align*}
where $X$ is an arbitrary $n \times n$ matrix, $\|\cdot\|_2$ is Hilbert-Schmidt norm and $\nu>1$. We also establish a Heinz type inequality involving
the Hadamard product of the form
\begin{align*}
2|||A^{1\over2}\circ B^{1\over2}|||\leq|||A^{s}\circ B^{1-t}+A^{1-s}\circ B^{t}|||
\leq\max\{|||(A+B)\circ I|||,|||(A\circ B)+I|||\},
\end{align*}
in which $s, t\in [0,1]$ and $|||\cdot|||$ is a unitarily invariant norm.
 \end{abstract} \maketitle
\section{Introduction and preliminaries}
Let ${\mathbb B}({\mathscr H})$ denote the $C^*$-algebra of all
bounded linear operators on a complex Hilbert space ${\mathscr H}$.
In the case when ${\rm dim}{\mathscr H}=n$, we identify ${\mathbb
B}({\mathscr H})$ with the  matrix algebra $\mathbb{M}_n$ of all
$n\times n$ matrices with entries in the complex field $\mathbb{C}$.
An operator $A\in{\mathbb B}({\mathscr H})$ is called positive
(positive semidefinite for matrices) if $\langle Ax,x\rangle\geq0$
for all $x\in{\mathscr H }$. The set of all positive invertible operators (respectively, positive definite matrices) is denoted by ${\mathbb B}({\mathscr H})_{++}$ (respectively, $\mathcal{P}_n$ ).

The Gelfand map $f(t)\mapsto f(A)$ is an
isometrically $*$-isomorphism between the $C^*$-algebra
$C(\sigma(A))$ of all continuous functions on the spectrum
$\sigma(A)$ of a selfadjoint operator $A$ and the $C^*$-algebra
generated by $A$ and the identity operator $I$ such that If $f, g\in
C(\sigma(A))$, then
$f(t)\geq g(t)\,\,(t\in\sigma(A))$ implies that $f(A)\geq g(A)$.

If $\{e_j\}$ is an orthonormal basis of ${\mathscr H}$, $V:{\mathscr H}\to
{\mathscr H}\otimes{\mathscr H}$ is the isometry defined by
$Ve_j=e_j\otimes e_j$ and $A\otimes B$
is the tensor product of operators $A, B$, then the Hadamard product $A\circ
B$ regarding to $\{e_j\}$ is expressed by $A\circ B=V^*(A\otimes B)V$.

A unitarily invariant norm $|||\cdot|||$ is defined on a norm ideal
$\mathfrak{L}_{|||\cdot|||}$ of ${\mathbb B}({\mathscr H})$
associated with it and has the property $|||UXV|||=|||X|||$, where
$U$ and $V$ are arbitrary unitaries in ${\mathbb B}({\mathscr H})$ and $X\in\mathfrak{L}_{|||\cdot|||}$. A compact operator
$A\in{\mathbb B}({\mathscr H})$ is called Hilbert-Schmidt if $\|A\|_2=\left(\sum_{j=1}^{\infty}s_j^2(A)\right)^{1/2}<\infty$,
where $s_1(A), s_2(A), \cdots$ are the singular values of $A$, i.e., the eigenvalues
of the positive operator $|A|=(A^*A)^{1\over2}$ enumerated as
$s_1(A)\geq s_2(A)\geq\cdots$ with their multiplicities counted. The
Hilbert-Schmidt norm is a unitarily invariant norm. For
$A=[a_{ij}]\in\mathbb{M}_n$, it holds that $\|A\|_2=\Big{(}\sum_{i,j=1}^n|a_{i,j}|^2\Big{)}^{1/2}$.
For two operators $A, B\in{\mathbb B}({\mathscr H})_{++}$, let $A\sharp_\mu B=A^{\frac{1}{2}}\left(A^{\frac{-1}{2}}BA^{\frac{-1}{2}}\right)^{\mu}A^{\frac{1}{2}}\,\,(\mu\in\mathbb{R})$. The operators
$A\sharp_{\frac{1}{2}} B$ and  $A\nabla B={A+B\over2}$ are called the operator geometric mean and the operator arithmetic mean, respectively.

The Heinz mean is defined by
\begin{align*}
H_\nu(a,b)={a^\nu b^{1-\nu}+a^{1-\nu}b^\nu\over 2}\qquad(0\leq \nu \leq1,\, a,b >0).
 \end{align*}
 The function $H_\nu$ is symmetric about the point $\nu={1\over2}$. Note that $H_0(a,b)=H_1(a,b)={a+b\over2}$,
$H_{1/2}(a,b)=\sqrt{ab}$ and $H_{1/2}(a,b)\leq H_\nu(a,b)\leq H_0(a,b)$
 for all $\nu\in[0,1]$.

The Heinz norm (double) inequality, which is one of the essential inequalities in operator theory, states that for any positive operators $A, B\in
{\mathbb B}({\mathscr H})$, any operator $X \in {\mathbb B}({\mathscr H})$ and
any $\nu\in[0,1]$, the double inequality
\begin{align}\label{purpos}
2\|A^{1\over2}XB^{1\over2}\|\leq \|A^\nu XB^{1-\nu}+A^{1-\nu}XB^\nu\|\leq \|AX+XB\|
\end{align}
holds; see \cite{13}. Bhatia and Davis \cite{davis} proved that \eqref{purpos} is valid for any unitarily invariant norm.
Fujii et al. \cite{1010} proved
that the right hand side inequality at \eqref{purpos} is equivalent to several other norm
inequalities such as\\
(i) the McIntosh inequality \cite{79.AM} asserting that
$\|A^*AX+XB^*B\| \ge 2\|AXB^*\|$ for all $A,B,X\in {\mathbb B}({\mathscr H})$;\\
(ii) the Corach--Porta--Recht inequality
$\|AXA^{-1}+A^{-1}XA\|\geq 2\|X\|$,
where $A\in {\mathbb B}({\mathscr H})$ is selfadjoint and invertible and $X\in {\mathbb B}({\mathscr H})$ (see also \cite{ms2}), and\\
(iii) the inequality
$\|A^{2m + n}XB^{-n}+A^{-n}XB^{2m+n}\|\geq \|A^{2m}X+XB^{2m}\|$ in which $A,B$ are invertible self-adjoint operators, $X$ is an arbitrary operator in $\mathbb{B}(\mathscr{H})$ and both $m$ and $n$ are
nonnegative integers; see also Section 3.9 of the  monograph \cite{new}.

Audenaert \cite{aud} gave a singular value inequality for the Heinz means of matrices as follows:
If $A, B\in \mathbb{M}_n$ are positive semidefinite and $\nu\in[0,1]$, then
\begin{align*}
s_j(A^\nu B^{1-\nu}+A^{1-\nu}B^\nu)\leq s_j(A+B).
 \end{align*}
Kittaneh and Manasrah \cite{young2} showed a refinement of the
right hand side of inequality \eqref{purpos} for the Hilbert-Schmidt norm
as follows:
\begin{align}\label{r099}
\|A^\nu XB^{1-\nu}+A^{1-\nu}X B^\nu\|_2^2+2r_0\|AX-XB\|_2^2\leq\| AX+XB\|_2^2,
 \end{align}
in which $A, B, X\in \mathbb{M}_n$ such that $A, B$ are
positive semidefinite, $\nu\in[0,1]$ and $r_0=\min\{\nu, 1-\nu\}$.
Kaur et al. \cite{mos}, by using the convexity of the function
$f(\nu)=|||A^{1-\nu} XB^\nu+A^\nu XB^{1-\nu}|||\,\,(\nu\in[0,1])$
presented more refinements of the Heinz inequality. More precisely,
for $A, B, X\in \mathbb{M}_n$ such that $A, B$ are positive
semidefinite and $\nu\in[0,1]$, they showed the inequality
\begin{align*}
|||A^\nu XB^{1-\nu}+A^{1-\nu}X B^\nu|||\leq|||4r_1A^{1\over2}XB^{1\over2}+(1-2r_1)(AX+XB)|||,
 \end{align*}
where $r_1=\min\left\{\nu, |{1\over2}-\nu|, 1-\nu\right\}$.
 It is shown in \cite{manas2} a reverse of inequality \eqref{r099} as
\begin{align}\label{momo1}
\| AX+XB\|_2^2\leq\|A^\nu XB^{1-\nu}+A^{1-\nu}X B^\nu\|_2^2+2r_0\|AX-XB\|_2^2,
 \end{align}
where $A, B, X\in \mathbb{M}_n$ such that $A, B$ are
positive semidefinite, $\nu\in[0,1]$ and $r_0=\max\{\nu, 1-\nu\}$. Aujla \cite{Aujla1} showed that
\begin{align*}
 2|||A^{\frac{1}{2}}XB^{\frac{1}{2}}|||\leq|||A^sXB^{1-t}+A^{1-s}XB^t|||,
 \end{align*}
 where $A, B, X\in \mathbb{M}_n$ such that $A, B$ are positive semidefinite, $s,t\in[0,1]$. It is remarkable that, by using the fact that the function $g(s,t)=|||A^sXB^{1-t}+A^{1-s}XB^t|||$ attains its maximum at the vertices of the square $[0,1]\times [0,1]$, one can see that under the same conditions as above
 \begin{align*}
 |||A^sXB^{1-t}+A^{1-s}XB^t|||\leq \max\left\{|||AX+XB|||, |||AXB+X|||\right\},
 \end{align*}
Recently, Krni\'c et al. used the Jensen functional to improve several Heinz type inequalities \cite{KP}.

In this paper, we obtain a reverse of \eqref{r099} and some other operator inequalities. We also show some
results on the Hadamard product. In particular, we get the following Heinz type inequality
\begin{align*}
2|||A^{1\over2}\circ B^{1\over2}|||\leq|||A^{s}\circ B^{1-t}+A^{1-s}\circ B^{t}|||
\leq\max\{|||(A+B)\circ I|||,|||(A\circ B)+I|||\},
\end{align*}
where $A, B\in\mathcal{P}_n, X\in \mathbb{M}_n$ and $s,t\in[0,1]$.
\section{A reverse of the Heinz inequality for matrices}

In this section, we present a converse of the Heinz inequality and
give several refinements for matrices.
\begin{lemma}\label{lemma11}
Let $a, b>0$ and $\nu\not\in[0,1]$. Then
\begin{align}\label{004}
 a+b\leq a^{\nu}b^{1-\nu}+b^{\nu}a^{1-\nu}.
 \end{align}
\end{lemma}
\begin{proof}
Let $\nu\not\in[0,1]$.
Assume that  $f(t)=t^{1-\nu}-\nu+(\nu-1)t\,\,(t\in(0,\infty))$. It is easy to see that
$f(t)$ has a minimum at $t=1$ in the interval $(0,\infty)$. Hence
$f(t)\geq f(1)=0$ for all $t>0$. Assume that $a,b>0$.  Letting $t={b\over a}$, we get
\begin{align}\label{e53}
 \nu a+(1-\nu)b\leq a^{\nu}b^{1-\nu}.
 \end{align}
Applying \eqref{e53} we obtain
\begin{align*}
 \nu a+(1-\nu)b\leq a^{\nu}b^{1-\nu}\,\,\, \textrm{and}\,\,\,\nu b+(1-\nu)a\leq b^{\nu}a^{1-\nu},
 \end{align*}
whence
 \begin{align*}
 a+b\leq a^{\nu}b^{1-\nu}+b^{\nu}a^{1-\nu}.
 \end{align*}
\end{proof}

For $\nu\not\in[0,1]$, if we replace $\nu$ by $\nu/(2\nu-1)$ and $A, B$,$ X$ by  $A^{2\nu-1}, B^{2\nu-1}$,$ A^{1-\nu}XB^{1-\nu}$ in \eqref{purpos}, respectively, then we reach the following Theorem, complementary to the right
inequality in \eqref{purpos}.

\begin{theorem}\label{tm_hein_compl} Let $A, B\in \mathcal{P}_n$, $X\in \mathbb{M}_n$
and  $\nu\not\in[0,1]$. Then
\begin{equation*}
\left|\left|\left|AX+XB \right|\right|\right|\leq \left|\left|\left|A^{\nu}XB^{1-\nu}+A^{1-\nu}XB^{\nu} \right|\right|\right|.
\end{equation*}
\end{theorem}
In the next theorem we show a reverse of \eqref{r099}. First, we need the following lemma.
\begin{lemma}\label{lemma14}
Let $a, b>0$ and $\nu\not\in[\frac{1}{2},1]$. Then
\begin{itemize}\label{reverse}
\item[(i)]$\nu a+(1-\nu)b+(\nu-1)(\sqrt{a}-\sqrt{b})^2\leq a^{\nu}b^{1-\nu}$
 \item[(ii)]$(a+b)+2(\nu-1)(\sqrt{a}-\sqrt{b})^2\leq a^{\nu}b^{1-\nu}+b^{\nu}a^{1-\nu}$
\item[(iii)]$(a+b)^2+2(\nu-1)({a}-{b})^2\leq (a^{\nu}b^{1-\nu}+b^{\nu}a^{1-\nu})^2$.
\end{itemize}
 \end{lemma}
\begin{proof}
Let $a, b>0$ and $\nu\not\in[\frac{1}{2},1]$.\\
$(\textrm{i})$ By inequality \eqref{e53},
\begin{align*}
\nu a+(1-\nu)b+(\nu-1)(\sqrt{a}-\sqrt{b})^2&=(2-2\nu)\sqrt{ab}+(2\nu-1)a\\&\leq(\sqrt{ab})^{2-2\nu}a^{2\nu-1}= a^{\nu}b^{1-\nu}.
 \end{align*}
 $(\textrm{ii})$ It can be proved in a similar fashion as $(\textrm{i})$.\\
 $(\textrm{iii})$ It follows from $(\textrm{ii})$ by replacing  $a$ by $a^2$ and $b$ by $b^2$.
 \end{proof}

\begin{theorem}\label{235}
Suppose that $A, B\in\mathcal{P}_n, X\in\mathbb{M}_n$  and $\nu>1$. Then
\begin{align*}
\|AX+XB\|_2^2+ 2(\nu-1) \|AX-XB\|_2^2\leq \|A^{\nu}XB^{1-\nu}+A^{1-\nu}XB^{\nu}\|_2^2\,.
 \end{align*}
\end{theorem}
\begin{proof}
By the spectral decomposition \cite[Theorem 3.4]{Zhang1}, there are unitary matrices $U, V\in\mathbb{M}_n$ such that $A=U\Lambda U^*$ and $B=V\Gamma V^*$, where $\Lambda={\rm diag}(\lambda_1, \lambda_2, \cdots, \lambda_n)$, $\Gamma={\rm diag}(\gamma_1, \gamma_2, \cdots, \gamma_n)$, and $\lambda_j, \gamma_j\,\,(j=1, \cdots, n)$ are eigenvalues of $A, B$, respectively. These numbers are positive. If $Z=U^*XV=\big{[}z_{ij}\big{]}$, then
\begin{align}\label{225}
AX+XB=U\Big{(}\Lambda Z+Z\Gamma\Big{)}V^*=U\Big{[}\Big{(}\lambda_i+\gamma_j\Big{)}z_{ij}\Big{]}V^*,
 \end{align}
 \begin{align}\label{2260}
AX-XB=U\Lambda U^*X-XV\Gamma V^*=
U\Big{[}\Lambda Z-Z\Gamma\Big{]}V^*=
U\Big{[}\Big{(}\lambda_i -\gamma_j\Big{)}z_{ij}\Big{]}V^*
 \end{align}
 and
\begin{align}\label{227}
A^{\nu}XB^{1-\nu}+A^{1-\nu}XB^{\nu}&=U\Lambda^{\nu}U^*XV\Gamma^{1-\nu}V^*
+U\Lambda^{1-\nu}U^*XV\Gamma^{\nu}V^*\nonumber\\&=
U\Lambda^{\nu}Z\Gamma^{1-\nu}V^*
+U\Lambda^{1-\nu}Z\Gamma^{\nu}V^*\nonumber\\&=
U\Big{[}\Lambda^{\nu}Z\Gamma^{1-\nu}
+\Lambda^{1-\nu}Z\Gamma^{\nu}\Big{]}V^*\nonumber\\&=
U\Big{[}\Big{(}\lambda_i^{\nu}\gamma_j^{1-\nu}+\lambda_i^{1-\nu}\gamma_j^{\nu}\Big{)}z_{ij}\Big{]}V^*.
\end{align}
It follows from \eqref{225}, \eqref{2260} and \eqref{227} that
\begin{align*}
\|AX+XB\|_2^2&+2(\nu-1)\|AX-XB\|_2^2\\&=\sum_{i,j=1}^n
\Big{(}\lambda_i+\gamma_j\Big{)}^2|z_{ij}|^2+2(\nu-1)\sum_{i,j=1}^n\Big{(}\lambda_i
-\mu_j\Big{)}^2|z_{ij}|^2
\\&\leq\sum_{i,j=1}^n\Big{(}\lambda_i^{\nu}\gamma_j^{1-\nu}+\lambda_i^{1-\nu}\gamma_j^{\nu}\Big{)}^2|z_{ij}|^2\,\,\,
\textrm{(by Lemma }\,\ref{lemma14}\,(\textrm{iii}))
\\&=\|A^{\nu}XB^{1-\nu}+A^{1-\nu}XB^{\nu}\|_2^2\,.
 \end{align*}
\end{proof}
\begin{remark}
Utilizing Lemma \ref{reverse}, one can easily see that Theorem \ref{235} holds for $\nu<\frac{1}{2}$. The case $\nu<\frac{1}{2}$ is not interesting since the left hand side is less precise than the left hand side of Theorem \ref{tm_hein_compl}, but the case of $0\leq \nu\leq \frac{1}{2}$ coincides with  inequality \eqref{momo1}.
\end{remark}
 Theorem \ref{235} yields the next two corollaries.
\begin{corollary}\label{coro21}
Suppose that $A, B\in\mathcal{P}_n, X\in\mathbb{M}_n$  and $\nu>1$. Then
\begin{align*}
\|AX+XB\|_2=\|A^{\nu}XB^{1-\nu}+A^{1-\nu}XB^{\nu}\|_2
 \end{align*}
 if and only if $AX=XB$.
\end{corollary}
\begin{proof}
If $AX=XB$, then $A^{\nu}X=XB^{\nu}$ and $A^{1-\nu}X=XB^{1-\nu}$. Hence
\begin{align*}
\|A^{\nu}XB^{1-\nu}+A^{1-\nu}XB^{\nu}\|_2
=\|A^{\nu}A^{1-\nu}X+XB^{1-\nu}B^{\nu}\|_2=\|AX+XB\|_2.
 \end{align*}
 Conversely, assume that $\|AX+XB\|_2=\|A^{\nu}XB^{1-\nu}+A^{1-\nu}XB^{\nu}\|_2$. It follows from Theorem \ref{235} that $\|AX-XB\|_2=0$. Thus $AX=XB$.
\end{proof}
\begin{corollary}
Let $A, B\in\mathcal{P}_n$  and $\nu>1$. Then
\begin{align*}
s_j(A+B)=s_j(A^{\nu}B^{1-\nu}+A^{1-\nu}B^{\nu})\qquad(j=1,2,\cdots, n)
 \end{align*}
 if and only if $A=B$.
\end{corollary}
\begin{proof}
If $A=B$, then $A+B=A^{\nu}B^{1-\nu}+A^{1-\nu}B^{\nu}$. Conversely, assume
that $s_j(A+B)=s_j(A^{\nu}B^{1-\nu}+A^{1-\nu}B^{\nu})\,\,(j=1,2,\cdots,
n)$. Then $\|AX+XB\|_2=\|A^{\nu}XB^{1-\nu}+A^{1-\nu}XB^{\nu}\|_2$. It
follows from Corollary \ref{coro21} that $A=B$.
\end{proof}

\section{A reverse of the Heinz inequality for operators }

In this section we obtain a reverse of the Heinz inequality for two
positive invertible operators as well as some other operator
inequalities.\\ In \cite{debrecenkit}, the authors investigated an operator version of
the classical Heinz mean, i.e., the operator
\begin{equation}
\label{Heinz_operator} H_\nu(A,B)=\frac{A\ \! \sharp_\nu \ \! B+
A\ \! \sharp_{1-\nu} \ \! B}{2},
\end{equation}where $A, B\in\mathbb{B}(\mathscr{H})_{++}$, and $\nu\in [0,1]$.
As in the real case, this mean interpolates between
arithmetic and geometric mean, that is,
\begin{equation*}
\label{heinz_interpolate} A\ \!\sharp \ \! B\leq H_\nu(A,B)\leq A\
\!\nabla \ \! B.
\end{equation*}
On the other hand, since $A, B\in\mathbb{B}(\mathscr{H})_{++}$,
the expression (\ref{Heinz_operator}) is also well-defined for
$\nu\not\in [0,1]$. Using inequality \eqref{e53} and the functional calculus for $A^{-1\over2}BA^{-1\over2}$ we get the following result.

 \begin{equation}\label{gbz}
H_{1-\nu}(A,B)=\frac{A\sharp_{1-\nu}B+A\sharp_{\nu}B}{2}\geq\frac{A\nabla_{1-\nu}B+A\nabla_{\nu}B}{2}=A\nabla B,
\end{equation}
where $A, B\in{\mathbb B}({\mathscr H})_{++}$ and $\nu\not\in [0,1]$.
Applying Lemma \ref{lemma14} $(\textrm{ii})$, we have a refinement
of inequality \eqref{gbz}.
\begin{theorem}\label{nege}
Let $A, B\in{\mathbb B}({\mathscr H})_{++}$ and $\nu>1$. Then
 \begin{align*}
A\nabla B+2(\nu-1)(A\nabla B-A\sharp_{1/2}B)\leq H_{1-\nu}(A,B)\,.
 \end{align*}
\end{theorem}
\begin{proof}
By Lemma \ref{lemma14} $(\textrm{ii})$, we have
$\frac{1+t}{2}+(\nu-1)(t-2\sqrt{t}+1)\leq \frac{t^{1-\nu}+t^{\nu}}{2}\,\,(t>0)$. Hence
\begin{align}\label{232}
\frac{(1+A^{-{1\over2}}BA^{-{1\over2}})}{2}+(\nu-1)(A^{-{1\over2}}BA^{-{1\over2}}-2(A^{-{1\over2}}BA^{-{1\over2}})^{1\over2}+1)\nonumber\\
\leq \frac{ (A^{-{1\over2}}BA^{-{1\over2}})^{1-\nu}+(A^{-{1\over2}}BA^{-{1\over2}})^{\nu}}{2}.
 \end{align}
 Multiplying $A^{1\over2}$ by the both sides of \eqref{232} we get
 \begin{align*}
A\nabla B+2(\nu-1)(A\nabla B-A\sharp_{1/2}B)\leq \frac {A\sharp_{1-\nu}B+A\sharp_{\nu}B}{2}= H_{1-\nu}(A,B)\,.
 \end{align*}
\end{proof}
\begin{remark}
Theorem \ref{nege} also holds for $\nu<\frac{1}{2}$. The case when $\nu<\frac{1}{2}$ is not interesting, since it is less precise than inequality \eqref{gbz}, but the case of $0\leq \nu\leq\frac{1}{2}$ coincides with the inequality at \cite[Corollary 2]{debrecenkit}.
\end{remark}
Applying Theorem \ref{nege} we get immediately the following result.
\begin{corollary}
Let $A, B\in{\mathbb B}({\mathscr H})_{++}$ and $\nu>1$. Then
\begin{align*}
H_{1-\nu}(A,B)=A\nabla B
 \end{align*}
 if and only if $A=B$.
\end{corollary}

Applying Lemma \ref{lemma11} we get
\begin{align*}
a+a^{-1}\leq a^{\nu}+ a^{-{\nu}}\qquad(a>0,\,\nu>1).
\end{align*}
Utilizing this inequality, the functional calculus for $A\otimes B^{-1}$ and the
definition of the Hadamard product we get the following result.
\begin{proposition}
Let $A, B\in{\mathbb B}({\mathscr H})_{++}$ and $\nu>1$. Then
\begin{itemize}
\item[(i)]$A\otimes B^{-1}+A^{-1}\otimes B\leq A^{\nu}\otimes B^{-{\nu}}+A^{-{\nu}}\otimes B^{\nu}$
 \item[(ii)]$A\circ B^{-1}+A^{-1}\circ B\leq A^{\nu}\circ B^{-{\nu}}+A^{-{\nu}}\circ B^{\nu}$.
\end{itemize}
\end{proposition}
\section{Some  Heinz type inequality related to Hadamard product}
In this section, using some ideas of \cite{aujla} and \cite{Aujla1}, we show some Heinz type inequalities.
\begin{lemma}\cite[Theorem 1.1.3]{bha2}\label{a80}
Let $A, B\in\mathcal{P}_n$ and $X\in \mathbb{M}_n$. Then the block matrix
$\left(\begin{array}{cc}
 A&X\\
 X^*&B
 \end{array}\right)$
 is positive semidefinite if and only if $A\geq XB^{-1}X^*$.
\end{lemma}
\begin{theorem}\label{a82}
The two variables function
\begin{align*}
H(s,t)=A^{1+s}\otimes B^{1-t}+A^{1-s}\otimes B^{1+t}
\end{align*}
is convex on $[-1,1]\times[-1,1]$ and attains its minimum at $(0,0)$ for all $A, B\in\mathcal{P}_n$.
\end{theorem}
\begin{proof}
Since $H$ is continuous, it is enough to prove
 $$H(s_1,t_1)\leq{1\over2}(H(s_1+s_2,t_1+t_2)+H(s_1-s_2,t_1-t_2))$$
 for all $s_1\pm s_2,t_1\pm t_2\in[0,1]$; see \cite{Aujla1}.
 For $A, B\in\mathcal{P}_n$ and $s_1\pm s_2,t_1\pm t_2\in[0,1]$ it follows from Lemma \ref{a80} that the matrices
 $\left(\begin{array}{cc}
 A^{1+s_1+s_2}&A^{1+s_1}\\
 A^{1+s_1}&A^{1+(s_1-s_2)}
 \end{array}\right)$,
 $\left(\begin{array}{cc}
 A^{1-(s_1+s_2)}&A^{1-s_1}\\
 A^{1-s_1}&A^{1-(s_1-s_2)}
 \end{array}\right)$,

 $\left(\begin{array}{cc}
 B^{1+t_1+t_2}&B^{1+t_1}\\
 B^{1+t_1}&B^{1+(t_1-t_2)}
 \end{array}\right)$
 and
 $\left(\begin{array}{cc}
 B^{1-(t_1+t_2)}&B^{1-t_1}\\
 B^{1-t_1}&B^{1-(t_1-t_2)}
 \end{array}\right)$
 are positive semidefinite. Hence the matrices
 $$X=\left(\Small{\begin{array}{cc}
 A^{1+s_1+s_2}\otimes B^{1-(t_1+t_2)}+A^{1-(s_1+s_2)}\otimes B^{1+t_1+t_2} &A^{1+s_1}\otimes B^{1-t_1}+A^{1-s_1}\otimes B^{1+t_1}\\
 A^{1+s_1}\otimes B^{1-t_1}+A^{1-s_1}\otimes B^{1+t_1}& A^{1+(s_1-s_2)}\otimes B^{1-(t_1-t_2)}+A^{1-(s_1-s_2)}\otimes B^{1+(t_1-t_2)}
\end{array}}\right)$$
 is positive semidefinite. Similarly,
 $$Y=\left(\Small{\begin{array}{cc}
 A^{1+(s_1-s_2)}\otimes B^{1+(t_1-t_2)}+A^{1-(s_1-s_2)}\otimes B^{1-(t_1-t_2)} &A^{1+s_1}\otimes B^{1-t_1}+A^{1-s_1}\otimes B^{1+t_1}\\ A^{1+s_1}\otimes B^{1-t_1}+A^{1-s_1}\otimes B^{1+t_1}&A^{1+s_1+s_2}\otimes B^{1-(t_1+t_2)}+A^{1-(s_1+s_2)}\otimes B^{1+t_1+t_2}
\end{array}}\right)$$
is positive semidefinite. Thus $$X+Y=\left(\Small{\begin{array}{cc}
 H(s_1+s_2,t_1+t_2)+H(s_1-s_2,t_1-t_2)&2H(s_1,t_1)\\
 2H(s_1,t_2)&H(s_1+s_2,t_1+t_2)+H(s_1-s_2,t_1-t_2)
 \end{array}}\right)$$
 is positive semidefinite and therefore
 $$\left(\begin{array}{cc}
 I_n&-I_n\\
 0&0
 \end{array}\right)(X+Y)\left(\begin{array}{cc}
 I_n&0\\
 -I_n&0
 \end{array}\right)$$
 is positive semidefinite. Hence $H(s_1+s_2,t_1+t_2)+H(s_1-s_2,t_1-t_2)-2H(s_1,t_1)\geq0$,
 which proves the convexity of $H$. Further note that $H(s,t)=H(-s,-t)\,\,s,t\in[0,1]$. This together with
the convexity of $H$ imply that $H$ attains its minimum at $(0,0)$.
\end{proof}
If in Theorem \ref{a82} we replace $s$, $t$, $A, B$  by ${2s-1}$, ${2t-1}$, $A^{1\over2}, B^{1\over2}$, respectively, we reach the following result.
\begin{corollary}
The two variables function
\begin{align*}
K(s,t)=A^{s}\circ B^{1-t}+A^{1-s}\circ B^{t}\,\,\,( A, B\in\mathcal{P}_n)
\end{align*}
is convex on $[0,1]\times[0,1]$ and attains its minimum at $({1\over2},{1\over2})$.
\end{corollary}
Aujla et al. \cite{aujla} showed that
 \begin{align*}
2|||A^{1\over2}\circ B^{1\over2}|||\leq|||A^{t}\circ B^{1-t}+A^{1-t}\circ B^{t}|||\leq|||A+B|||,
\end{align*}
where $A, B\in\mathcal{P}_n$ and $t\in[0,1]$. Now, we are ready to state our last result.
\begin{corollary}
Let $A, B\in\mathcal{P}_n$ and $s,t\in[0,1]$. Then
\begin{align*}
2|||A^{1\over2}\circ B^{1\over2}|||\leq|||A^{s}\circ B^{1-t}+A^{1-s}\circ B^{t}|||
\leq
\max\{|||(A +B)\circ I|||,|||(A\circ B)+I|||\}.
\end{align*}
\end{corollary}
\begin{proof}
Let $K(s,t)=A^{s}\circ B^{1-t}+A^{1-s}\circ B^{t}$. If we put $G(s,t)=|||K(s,t)|||$, then by the convexity of $K$ and Fan Dominance Theorem \cite[p. 58]{bha2} (see also \cite{MOS}), the function $G$ is convex on $[0,1]\times[0,1]$, and attains minimum at $({1\over2},{1\over2})$. Hence we have the first inequality. In addition, since the function $G$ is continuous and convex on $[0,1]\times[0,1]$, it follows that $G$ attains its maximum at the vertices of the square. Moreover, due to the symmetry there are only two possibilities for the maximum.
\end{proof}

\bibliographystyle{amsplain}

\end{document}